\documentclass[11pt]{amsart}

\usepackage{amsfonts,amsthm,amsmath,amssymb}
\usepackage{amscd}
\usepackage[comma,sort,compress, numbers]{natbib}

\usepackage{hyperref} 
\hypersetup{ colorlinks=true, linkcolor=blue, citecolor=blue,
  filecolor=blue, urlcolor=blue}

\newcommand{\weak}{\stackrel{w}{\longrightarrow}}
\newcommand{\prob}{\stackrel{P}{\longrightarrow}}

\newcommand{\one}{{\bf 1}}

\newcommand{\reals}{{\mathbb R}}
\newcommand{\bbr}{\reals}

\newcommand{\bbz}{\protect{\mathbb Z}}

\newtheorem{theorem}{Theorem}[section]
\newtheorem{fact}{Fact}[section]
\newtheorem{lemma}{Lemma}[section]

\newtheorem{remark}{Remark}
\newtheorem{question}{Question}
\newtheorem{prop}{Proposition}[section]

\def\ol{\overline}
\def\chs{\citet{chakrabarty:hazra:sarkar:2014}}

\def\Var{{\rm Var}}
\def\E{{\rm E}}
\def\esd{{\rm ESD}}

\DeclareMathOperator{\Tr}{Tr}

\begin{document}

\bibliographystyle{abbrvnat}

\title[Free multiplicative convolution]{Remarks on absolute continuity in the context of free probability and random matrices}

\author[A. Chakrabarty]{Arijit Chakrabarty}\thanks{The research of both authors is supported by their respective INSPIRE grants from the Department of Science and Technology, Government of India}
\address{Theoretical Statistics and Mathematics Unit, Indian Statistical
Institute, New Delhi}
\email{arijit@isid.ac.in}

\author[R. S. Hazra]{Rajat Subhra Hazra}
\address{Theoretical Statistics and Mathematics Unit, Indian Statistical
Institute, Kolkata}
\email{rajatmaths@gmail.com}


\begin{abstract}
 In this note, we show that the limiting spectral distribution of symmetric random matrices with stationary entries is absolutely continuous under some sufficient conditions. This result is applied to obtain sufficient conditions on a probability measure for its free multiplicative convolution with the semicircle law to be absolutely continuous.
\end{abstract}

\maketitle

\section{Introduction}\label{intro}

For two probability measures $\mu$ and $\nu$ on $\mathbb R$, one can
associate the free additive convolution $\mu\boxplus\nu$. This is defined as the
distribution of $X_\mu+Y_\nu$ where $X_\mu$ and $Y_\nu$ are self-adjoint
variables affiliated to a tracial $W^*$- probability space and are free from
each other. Similarly, for probability measures $\mu$ and $\nu$ on $[0,\infty)$ and $\bbr$ respectively,
the free multiplicative convolution is denoted by $\mu\boxtimes\nu$ and
represents the law of $X_\mu^{1/2}Y_\nu X_\mu^{1/2}$ where $X_\mu$ and $Y_\nu$
are free variables as before. Since $\mu$ is supported on the positive half line, $X_\mu$ is positive. We refer to
\citet{bercovici:voiculescu:1993} and \citet{arizmendi:abreu:2009} for details of these
notions and some related transforms. The questions of absolute continuity of these convolutions, with respect to the Lebesgue measure, are important. For compactly supported and absolutely continuous
measures $\mu$ and $\nu$,
\citet{voiculescu1993} showed that $\mu\boxplus\nu$ is absolutely continuous.
The result was extended to the non-compactly supported case when one of the measures is the semicircle
law by~\citet{biane:1997}. Further regularity properties of additive convolution were studied by~\citet{belinschi2008}.  

Some recent works study regularity properties in the context of free multiplicative convolution ($\mu\boxtimes\nu$), see for example \citet{belinschi2005partially, zhong2013free, perez2012}. Theorem 4.1 of \citet{belinschi2003atoms} gives necessary and sufficient conditions for $\mu\boxtimes\nu$ to have an atom at a particular point. However, \emph{absolute} continuity  is much less understood. The following question is a step in that direction, namely, the multiplicative analogue of the problem addressed in \citet{biane:1997}.

\begin{question}\label{q1}
 Let $\mu$ be any probability measure on $[0,\infty)$, and let $\mu_s$ denote
the semicircle law, defined in \eqref{eq.defmus}. Under what conditions on $\mu$, is $\mu\boxtimes\mu_s$
absolutely continuous? 
\end{question}

Our first result gives a sufficient condition on $\mu$ to  answer Question~\ref{q1}.

\begin{theorem}\label{t1}
 Let $\mu$ be a probability measure on $\bbr$ such that
$\mu([\delta,\infty))=1$ for some $\delta>0$. Assume furthermore that $\mu$ has finite mean.
Then, $\mu\boxtimes\mu_s$ is absolutely continuous.
\end{theorem}

For stating the next question, we need to introduce a random matrix model. Let
$f$ be any non-negative integrable function on $[-\pi,\pi]^2$. Assume furthermore that $f$ is even, that is,
\begin{equation}\label{eq.even}
 f(-x,-y)=f(x,y),\,-\pi\le x,y\le\pi\,.
\end{equation}
Then, there
exists a mean zero stationary Gaussian process $(G_{i,j}:i,j\in\bbz)$ such that
\begin{equation}\label{intro.eq1}
 \E\left(G_{i,j}G_{i+u,j+v}\right)=\int_{-\pi}^\pi\int_{-\pi}^\pi
e^{\iota(ux+vy)}f(x,y)dx dy,\text{ for all }i,j,u,v\in\bbz\,,
\end{equation}
\eqref{eq.even} ensuring that the right hand side is real.
For $N\ge1$, let $\ol G_N$ be the $N\times N$ matrix defined by
\begin{equation}\label{intro.eq2}
 \ol G_N(i,j):=(G_{i,j}+G_{j,i})/\sqrt N,\,1\le i,j\le N\,.
\end{equation}
Above and elsewhere, for any matrix $H$, $H(i,j)$ denotes its $(i,j)$-th entry.
It has been shown in Theorem 2.1 of \chs~ that there exists a (deterministic) probability measure
$\nu_f$ such that 
\begin{equation}\label{intro.eq3}
 \esd(\ol G_N)\to\nu_f\,,
\end{equation}
weakly in probability, as $N\to\infty$. $\esd$ stands for the empirical spectral distribution
for a symmetric $N\times N$ random matrix $H$, which is a random probability measure on
$\bbr$ defined by 
\[
 \left(\esd(H)\right)(\cdot):=\frac1N\sum_{j=1}^N\delta_{\lambda_j}(\cdot)\,,
\]
where $\lambda_1\le\ldots\le\lambda_N$ are the eigenvalues of $H$, counted with multiplicity.

The second question that this paper attempts to answer is the following.

\begin{question}\label{q2}
 Under what conditions on $f$, is $\nu_f$ absolutely continuous?
\end{question}

The answer is provided by the following result.

\begin{theorem}\label{t2}
 If $${ess\inf}_{(x,y)\in[-\pi,\pi]^2}\left[f(x,y)+f(y,x)\right]>0\,,$$ then $\nu_f$ is absolutely continuous, where 
``$ess\inf$'' denotes the essential infimum.
\end{theorem}

While Questions \ref{q1} and \ref{q2} seem unrelated a priori, the reader will notice after seeing the proofs that they are not so. This is because random matrix theory is used as a tool for proving Theorem \ref{t1}, Theorem \ref{t2} being anyway a question about a random matrix. 
It is shown in Proposition~\ref{prop:sum:mult}, which is a consequence of Theorem~\ref{t2}, that if a measure satisfies the conditions of Theorem~\ref{t1} then,
its free multiplicative convolution with a semicircle law is also the free additive convolution of another measure and a dilated semicircle law.
The proofs are compiled in the following section. Many known facts are used, which are collected in Section~\ref{sec: appendix} for the convenience of the reader.

\section{Proofs}\label{sec:proof}

For the proof of the results we shall refer to various known facts which are
listed in the Appendix of this article. One of the main ingredients is the
following fact which follows from Proposition 22.32, page 375 of
\citet{nica:speicher:2006}. The latter result reasserts the seminal discovery in \citet{voiculescu1991limit} that the Wigner matrix is
asymptotically freely independent of a deterministic matrix which has a compactly supported limiting spectral distribution.

\begin{fact}\label{f1}
 Assume that for each $N$, $A_N$ is a $N\times N$ Gaussian Wigner matrix scaled
by $\sqrt N$, that is, $(A_N(i,j):1\le i\le j\le N)$ are i.i.d. normal random
variables with mean zero and variance $1/N$, and $A_N(j,i)=A_N(i,j)$.
 Suppose that $B_N$ is a $N\times N$ random matrix, such that as $N\to\infty$,
\begin{equation}\label{f1.assume}
 \frac1N\Tr(B_N^k)\prob\int_\bbr x^k\mu(dx),\,k\ge1\,,
\end{equation}
for some compactly supported (deterministic) probability measure $\mu$. Furthermore, let the
families $(A_N:N\ge1)$ and $(B_N:N\ge1)$ be independent.
Then, as $N\to\infty$,
\[
 \frac1N\E_{\mathcal F}\Tr\left[(A_N+B_N)^k\right]\prob\int_\bbr
x^k\mu\boxplus\mu_s(dx)\text{ for all }k\ge1\,,
\]
where ${\mathcal F}:=\sigma(B_N:N\ge1)$ and $\E_{\mathcal F}$ denotes the
conditional expectation with respect to $\mathcal F$.
\end{fact}

We first proceed towards proving Theorem \ref{t2}. The first step in that direction is Lemma \ref{l1} below. However, before stating that, we define a dilated semicircle law $\mu_s(t)$ for all $t>0$. It is a probability measure on $\bbr$ given by

\begin{equation}\label{eq.defmus}
 \left(\mu_s(t)\right)(dx)=\frac{\sqrt{4t-x^2}}{2\pi t}\one(|x|\le2\sqrt t)\,dx,\,x\in\bbr\,.
\end{equation}
In other words, it is the semicircle law with variance $t$.
For $t=1$, it equals the standard semicircle law, that is, $\mu_s\equiv\mu_s(1)$.

\begin{lemma}\label{l1}
 Let $\alpha>0$. Denote
\[
 (f+\alpha)(\cdot,\cdot):=f(\cdot,\cdot)+\alpha\,.
\]
Then,
\begin{equation}\label{l1.claim}
 \nu_{f+\alpha}=\nu_f\boxplus\mu_s(8\pi^2\alpha)\,.
\end{equation}
\end{lemma}

\begin{proof} Let us first prove this for the case that $f$ is a trigonometric polynomial, that is, of the form
\begin{equation}\label{eq.trigpol}
 f(x,y)=\sum_{j,k=-n}^na_{j,k}e^{\iota(jx+ky)}\ge0\,,
\end{equation}
for some finite $n$ and real numbers $a_{j,k}$, in addition to satisfying \eqref{eq.even}.
 By Fact \ref{f6}, $\nu_{f+\alpha}$ and $\nu_f$ have compact supports, and hence
so does $\nu_f\boxplus\mu_s(8\pi^2\alpha)$. Therefore, it suffices to check that
\begin{equation}\label{l1.eq1}
 \int x^k\nu_{f+\alpha}(dx)=\int x^k(\nu_f\boxplus\mu_s(8\pi^2\alpha))(dx)\text{ for
all }k\ge1\,.
\end{equation}
Let $(G_{i,j}:i,j\in\bbz)$ be a mean zero stationary Gaussian process satisfying
\eqref{intro.eq1}. Let $(H_{i,j}:i,j\in\bbz)$ be a family of i.i.d.
$N(0,4\pi^2\alpha)$ random variables, independent of $(G_{i,j}:i,j\in\bbz)$. For
$N\ge1$, let $\ol G_N$ be as in \eqref{intro.eq2}, and further define the
$N\times N$ matrices $W_N$ and $Z_N$ by
\begin{eqnarray*}
 W_N(i,j)&:=&(H_{i,j}+H_{j,i})/\sqrt N,\,1\le i,j\le N,\\
 Z_N&:=&\ol G_N+W_N\,.
\end{eqnarray*}
Fact \ref{f6} implies that 
\[
 \frac1N\Tr\left(\ol G_N^k\right)\prob\int x^k\nu_f(dx),\,k\ge1\,,
\]
as $N\to\infty$. This, along with Fact \ref{f1} and the observation that the upper triangular entries of $W_N$ are i.i.d. $N(0,8\pi^2\alpha/N)$, implies that
\begin{equation}\label{l1.eq2}
 \frac1N\E_{\mathcal F}\Tr\left(Z_N^k\right)\prob\int
x^k(\nu_f\boxplus\mu_s(8\pi^2\alpha))(dx),\,k\ge1\,,
\end{equation}
where ${\mathcal F}:=\sigma(G_{i,j}:i,j\in\bbz)$.

It is easy to see that $(G_{i,j}+H_{i,j}:i,j\in\bbz)$ is a stationary mean zero
Gaussian process whose spectral density is $f+\alpha$, and hence Fact \ref{f6}
implies that 
\[
 \E\left[\frac1NE_{\mathcal
F}\Tr\left(Z_N^k\right)\right]=\E\left[\frac1N\Tr\left(Z_N^k\right)\right]
\to\int x^k\nu_{f+\alpha}(dx)\,,
\]
and
\[
 \Var\left[\frac1NE_{\mathcal
F}\Tr\left(Z_N^k\right)\right]\le\Var\left[\frac1N\Tr\left(Z_N^k\right)\right]
\to0\,,
\]
as $N\to\infty$, for all $k\ge1$.
Combining the above two limits and comparing with \eqref{l1.eq2} yields
\eqref{l1.eq1}, and completes the proof of \eqref{l1.claim} when $f$ is a trigonometric polynomial.

Now suppose that $f$ is just a non-negative integrable function satisfying \eqref{eq.even}. By considering the Fourier series of $\sqrt f$, one
can construct non-negative even trigonometric polynomials $f_n$ such that
\[
 f_n\to f\text{ in }L^1\,.
\]
Since \eqref{l1.claim} holds with $f$ replaced by $f_n$, it follows that
\begin{eqnarray*}
 \nu_{f_n+\alpha}&=&\nu_{f_n}\boxplus\mu_s(8\pi^2\alpha)\\
 &\weak&\nu_f\boxplus\mu_s(8\pi^2\alpha)\,,
\end{eqnarray*}
as $n\to\infty$,
the second line following from Fact \ref{f4} combined with Proposition 4.13 of
\citet{bercovici:voiculescu:1993}. Applying Fact \ref{f4} directly to
$\nu_{f_n+\alpha}$ yields
\begin{equation*}
 \nu_{f_n+\alpha}\weak\nu_{f+\alpha}\,,
\end{equation*}
as $n\to\infty$. This establishes the claimed identity and thus completes the proof.
\end{proof}

\begin{proof}[Proof of Theorem \ref{t2}]
 Define
\[
 g(x,y):=\frac12\left[f(x,y)+f(y,x)\right],\,-\pi\le x,y\le\pi\,.
\]
In view of Fact \ref{f5}, it suffices to show that $\nu_g$ is absolutely
continuous. The hypothesis implies that there exists $\alpha>0$ such that
$g\ge\alpha$ almost everywhere on $[-\pi,\pi]^2$. Define
\[
 h(\cdot,\cdot):=g(\cdot,\cdot)-\alpha\,.
\]
Notice that
\[
 \nu_f=\nu_g=\nu_h\boxplus\mu_s(8\pi^2\alpha)\,,
\]
the two equalities following by Fact \ref{f5} and Lemma \ref{l1} respectively. 
Fact \ref{f2} completes the proof.
\end{proof}

For proving Theorem \ref{t1}, we prove the following result which is of some
independent interest. This along with Fact \ref{f2} establishes Theorem
\ref{t1}.

\begin{prop}\label{prop:sum:mult}
 If $\mu$ satisfies the hypothesis of Theorem \ref{t1}, then there exists a
probability measure $\eta$ such that 
\[
 \mu\boxtimes\mu_s=\eta\boxplus\mu_s(\delta^2)\,.
\]
\end{prop}

\begin{proof}
 Define a function $r$ from $[-\pi,\pi]$ to $\bbr$ by
\[
 r(x):=\begin{cases}
 \frac1{2^{3/2}\pi}\inf\left\{y\in\bbr:\frac{|x|}{\pi}\le
\mu(-\infty,y]\right\},&0<|x|<\pi\,,\\
0,&\text{otherwise}\,.
\end{cases}
\]
Let $U$ be an Uniform$(-\pi,\pi)$ random variable. Clearly, 
\begin{equation}\label{neweq1}
 P(2^{3/2}\pi r(U)\in\cdot)=\mu(\cdot)\,,
\end{equation}
which implies that
\[
 \int_{-\pi}^\pi r(x)dx=2\pi\E[r(U)]=2^{-1/2}\int_0^\infty
x\mu(dx)<\infty\,.
\]
The hypothesis that $\mu(-\infty,\delta)=0$ implies that 
\[
 r(x)\ge\frac\delta{2^{3/2}\pi},\,0<|x|<\pi\,.
\]
Defining 
\[
 f(x,y):=r(x)r(y)\,,
\]
it follows that $f$ is an even (as in \eqref{eq.even}) non-negative integrable function bounded below by
$\alpha$ almost everywhere, where
\[
 \alpha:=\frac{\delta^2}{8\pi^2}\,.
\]
Fact \ref{f3} and \eqref{neweq1} imply that
\begin{eqnarray*}
 \mu\boxtimes\mu_s&=&\nu_f\\
 &=&\nu_{f-\alpha}\boxplus\mu_s(8\pi^2\alpha)\,,
\end{eqnarray*}
the second equality following from Lemma \ref{l1}. Setting
$\eta:=\nu_{f-\alpha}$, this completes the proof.
\end{proof}

\begin{proof}[Proof of Theorem \ref{t1}]
 Follows from Proposition \ref{prop:sum:mult} and Fact \ref{f2}.
\end{proof}

\begin{remark}
 Theorem~\ref{t1} is similar in spirit to Theorem 4.1 of \cite{belinschi2003atoms}.
\end{remark}
\begin{remark}
Although finiteness of the mean of $\mu$ is used in proving Proposition~\ref{prop:sum:mult} and thereby Theorem~\ref{t1}, the authors conjecture that the results are true without that assumption. 
\end{remark}

\section{Appendix}\label{sec: appendix} 
In this section, we collect the various facts that have been used in the proofs in Section \ref{sec:proof}.

The following fact is Corollary 2 of \citet{biane:1997}.

\begin{fact}\label{f2}
For any probability measure $\mu$, $\mu\boxplus\mu_s$ is absolutely continuous
with respect to the Lebesgue measure.
\end{fact} 

The remaining facts are all quoted from \chs.

\begin{fact}[Theorem 2.4, \chs]\label{f3}
 Let $r$ be a non-negative integrable even function defined on $[-\pi,\pi]$, and
\[
 f(x,y):=r(x)r(y),\,-\pi\le x,y\le\pi\,.
\]
Then, 
\[
 \nu_f=\mu_r\boxtimes\mu_s\,,
\]
where $\nu_f$ is as in \eqref{intro.eq3}, $\mu_r$ is the law of $2^{3/2}\pi
r(U)$, and $U$ is a Uniform$(-\pi,\pi)$ random variable.
\end{fact}

\begin{fact}[Lemma 3.3, \chs]\label{f4}
 Suppose that for all $1\le n\le\infty$, $g_n$ is a non-negative, integrable and
even function on $[-\pi,\pi]^2$. By even, it is meant that $g_n(-x,-y)=g_n(x,y)$
for all $x,y$. If
\[
 g_n\to g_\infty\text{ in }L^1\text{ as }n\to\infty\,,
\]
then
\[
 \nu_{g_n}\weak\nu_{g_\infty}\,.
\]
\end{fact}

\begin{fact}[Lemma 3.5, \chs]\label{f5}
 If $f$ is a non-negative integrable even function on $[-\pi,\pi]^2$, and
\[
 g(x,y):=\frac12\left[f(x,y)+f(y,x)\right]\,,
\]
then
\[
 \nu_f=\nu_g\,.
\]
\end{fact}

The following fact has been proved in the course of proving Proposition 3.1 of
\chs; see (3.16) and (3.17) therein.

\begin{fact}\label{f6}
 Let $f$ be a non-negative even trigonometric polynomial on $[-\pi,\pi]^2$ as in \eqref{eq.trigpol}.
Let the matrix $\ol G_N$ be constructed as in \eqref{intro.eq2} using the random
variables $(G_{i,j})$ which are as in \eqref{intro.eq1}. Then, $\nu_f$ has
compact support, and for all $k\ge1$,
\begin{eqnarray*}
 \lim_{N\to\infty}\E\left[\frac1N\Tr(\ol G_N^k)\right]&=&\int_\bbr
x^k\nu_f(dx)\,,\\
 \text{and }\lim_{N\to\infty}\Var\left[\frac1N\Tr(\ol G_N^k)\right]&=&0\,.
\end{eqnarray*}
\end{fact}

\section*{Acknowledgement} The authors are grateful to Manjunath Krishnapur for helpful discussions and an anonymous referee for constructive comments.


\begin{thebibliography}{12}
\providecommand{\natexlab}[1]{#1}
\providecommand{\url}[1]{\texttt{#1}}
\expandafter\ifx\csname urlstyle\endcsname\relax
  \providecommand{\doi}[1]{doi: #1}\else
  \providecommand{\doi}{doi: \begingroup \urlstyle{rm}\Url}\fi

\bibitem[Arizmendi and P\'erez-Abreu(2009)]{arizmendi:abreu:2009}
O.~E. Arizmendi and V.~P\'erez-Abreu.
\newblock The {S}-transform of symmetric probability measures with unbounded
  support.
\newblock \emph{Proceedings of the American Mathematical Society}, 137\penalty0
  (9):\penalty0 3057--3066, 2009.

\bibitem[Belinschi(2003)]{belinschi2003atoms}
S.~T. Belinschi.
\newblock The atoms of the free multiplicative convolution of two probability
  distributions.
\newblock \emph{Integral Equations and Operator Theory}, 46\penalty0
  (4):\penalty0 377--386, 2003.

\bibitem[Belinschi(2008)]{belinschi2008}
S.~T. Belinschi.
\newblock The {L}ebesgue decomposition of the free additive convolution of two
  probability distributions.
\newblock \emph{Probability Theory and Related Fields}, 142\penalty0
  (1-2):\penalty0 125--150, 2008.

\bibitem[Belinschi and Bercovici(2005)]{belinschi2005partially}
S.~T. Belinschi and H.~Bercovici.
\newblock Partially defined semigroups relative to multiplicative free
  convolution.
\newblock \emph{International Mathematics Research Notices}, 2005\penalty0
  (2):\penalty0 65--101, 2005.

\bibitem[Bercovici and Voiculescu(1993)]{bercovici:voiculescu:1993}
H.~Bercovici and D.~Voiculescu.
\newblock Free convolution of measures with unbounded support.
\newblock \emph{Indiana University Mathematics Journal}, 42:\penalty0 733--773,
  1993.

\bibitem[Biane(1997)]{biane:1997}
P.~Biane.
\newblock On the free convolution with a semi-circular distribution.
\newblock \emph{Indiana University Mathematics Journal}, 46\penalty0
  (3):\penalty0 705--718, 1997.

\bibitem[Chakrabarty et~al.(2014)Chakrabarty, Hazra, and
  Sarkar]{chakrabarty:hazra:sarkar:2014}
A.~Chakrabarty, R.~S. Hazra, and D.~Sarkar.
\newblock From random matrices to long range dependence.
\newblock Available at \url{http://arxiv.org/pdf/1401.0780.pdf}, 2014.

\bibitem[Nica and Speicher(2006)]{nica:speicher:2006}
A.~Nica and R.~Speicher.
\newblock \emph{Lectures on the Combinatorics of Free Probability}.
\newblock Cambridge University Press, New York, 2006.

\bibitem[P{\'e}rez-Abreu and Sakuma(2012)]{perez2012}
V.~P{\'e}rez-Abreu and N.~Sakuma.
\newblock Free infinite divisibility of free multiplicative mixtures of the
  Wigner distribution.
\newblock \emph{Journal of Theoretical Probability}, 25\penalty0 (1):\penalty0
  100--121, 2012.

\bibitem[Voiculescu(1991)]{voiculescu1991limit}
D.~Voiculescu.
\newblock Limit laws for random matrices and free products.
\newblock \emph{Inventiones mathematicae}, 104\penalty0 (1):\penalty0 201--220,
  1991.

\bibitem[Voiculescu(1993)]{voiculescu1993}
D.~Voiculescu.
\newblock The analogues of entropy and of {F}isher's information measure in
  free probability theory, i.
\newblock \emph{Communications in mathematical physics}, 155\penalty0
  (1):\penalty0 71--92, 1993.

\bibitem[Zhong(2013)]{zhong2013free}
P.~Zhong.
\newblock On the free convolution with a free multiplicative analogue of the
  normal distribution.
\newblock To appear in {\it Journal of Theoretical Probability}, DOI:
  10.1007/s10959-014-0556-x, 2013.

\end{thebibliography}

\end{document}